\documentclass[sn-mathphys,Numbered,12pt]{sn-jnl}

\usepackage{graphicx}%
\usepackage{multirow}%
\usepackage{amsmath,amssymb,amsfonts}%
\usepackage{amsthm}%
\usepackage{mathrsfs}%
\usepackage[title]{appendix}%
\usepackage{xcolor}%
\usepackage{textcomp}%
\usepackage{manyfoot}%
\usepackage{booktabs}%
\usepackage{algorithm}%
\usepackage{algorithmicx}%
\usepackage{algpseudocode}%
\usepackage{listings}%

\theoremstyle{thmstyleone}%
\newtheorem{theorem}{Theorem}
\newtheorem{observation}[theorem]{Observation}%
\newtheorem{lemma}[theorem]{Lemma}
\newtheorem{conclusion}[theorem]{Conclusion}

\theoremstyle{thmstyletwo}%
\newtheorem{example}{Example}%

\theoremstyle{thmstylethree}%
\newtheorem{definition}{Definition}%

\newcommand{\mcl}{\mathcal}

\newcommand{\eps}{\varepsilon}

\newcommand{\N}{\mathbb{N}}
\newcommand{\R}{\mathbb{R}}
\newcommand{\PP}{\mathbb{P}}

\begin{document}

\title[Article Title]{On Exponential Convergence of Random Variables}


\author{\fnm{Dawid} \sur{Tar{\l}owski}}\email{dawid.tarlowski@gmail.com; dawid.tarlowski@im.uj.edu.pl, ORCID 0000-0002-6824-4568}



\affil{\orgdiv{Faculty of Mathematics and Computer Science}, \orgname{Jagiellonian University}, \orgaddress{\street{\L ojasiewicza 6}, \city{Krak\'ow}, \postcode{30-348}, \country{Poland}}}




\abstract{ Given the discrete-time sequence of nonnegative random variables,  general dependencies between the exponential convergence  of the expectations, exponential convergence of the trajectories  and the logarithmic growth of the corresponding expected hitting times are analysed.  The applications are presented: the general results are applied to the areas of optimization, stochastic control and estimation.}




\maketitle

\section{Introduction}\label{sec1}
 Various forms of exponential convergence of random variables appear naturally in many applications of Markov chains \cite{Markov}, supermartingales \cite{Liu},\cite{Mori}, and more general stochastic processes. The applications often involve various problems from optimization and control theory. This paper compares several definitions of exponential convergence under general assumptions, and next applies the results to some problems of optimization, control, and estimation.  Section \ref{SEC}  presents the general theory. It is assumed that $X_t\in[0,\infty)$  is some discrete-time sequence of nonnegative random variables, and most attention is paid to the exponential convergence of $X_t$ to zero, although some of the results may be applied to the exponential growth of $X_t$. Section\ref{SEC} shows, among other results, that the convergence rate of the expectations defined by
\begin{equation}\label{A}
A:=\limsup\limits_{t\to\infty}\sqrt[\leftroot{3} t]{E[X_t]}
\end{equation}
 bounds from above the convergence rate of the trajectories
\begin{equation}\label{E1}
\PP[\limsup\limits_{t\to\infty}\sqrt[\leftroot{3} t]{X_t}\leq A]=1
\end{equation}
and that, if $A\in(0,1)$,  the stopping moment $\tau_{\eps}=inf\{t\in\N\colon X_t<\eps\}$ is integrable and  satisfies
\begin{equation}\label{E2}
\limsup\limits_{\eps\to0^+} \frac{E[\tau_\eps]}{|\log(\eps)|}\leq\frac{-1}{\log(A)}.
\end{equation} 
The above statements are very general  and at the same time  the author of this paper did not manage to find the proof of \eqref{E1} and \eqref{E2} in the literature (the special case of $\eqref{E1}$ is proved in \cite{TAR} in the context of continuous optimization). Statement \eqref{E1} does not assume that $A\leq1$ and thus may be used to bound the exponential growth of the trajectories. It is worth to mention that in the literature the exponential decrease of $\PP[X_n\geq\eps]$ is often studied, for instance in the analysis of large deviations where one usually considers  the sequences of the form $X_n=|\frac{S_n-E[S_n]}{n}|$, \cite{Liu}, \cite{EC}. Generally, if for any $\eps>0$ the sequence $\PP[X_t>\eps],{t\in\N},$ decreases exponentially fast then $X_t\to0$ almost sure however the exponential decrease of the trajectories of $X_t$ is not forced. Note also that by Chebyshev inequality for any $\eps>0$ we have $E[X_t]\geq \eps\cdot\PP[X_t\geq\eps]$ 
which implies that  for any $\eps>0$, $\limsup\limits_{t\to\infty}\sqrt[\leftroot{3} t]{E[X_t]}\geq \limsup\limits_{t\to\infty} \sqrt[\leftroot{3} t]{\PP[X_t\geq\eps]}$ and hence:
\begin{equation}\label{E3}
A=\limsup\limits_{t\to\infty}\sqrt[\leftroot{3} t]{E[X_t]}\geq \lim\limits_{\eps\to0^+}\limsup\limits_{t\to\infty} \sqrt[\leftroot{3} t]{\PP[X_t\geq\eps]}.
\end{equation}
Thus, the exponential convergence of the expectations is a very strong convergence mode and the value of $A$ given by \eqref{A} bounds the convergence rate in case of other exponential convergence types \eqref{E1}, \eqref{E2}, \eqref{E3}. Section \ref{SEC} introduces general definitions precisely and proves  \eqref{E1},\eqref{E2} and related results.  Section \ref{SECA} applies those results to the areas of optimization and control. In the context of optimization  we usually deal with the sequences of the form $X_t=f(Z_t)$, where $Z_t$ is some sequence of random variables which represents an optimization process and $f$ represents some real-valued problem function. In such case it is an important question  how fast the sequence  $f(Z_t)$ approaches the global extremum of $f$ and all the above mentioned types of exponential convergence (linear convergence) are being considered, see \cite{Agapie}, \cite{Jager},\cite{Jager2},\cite{Mori},\cite{R1},\cite{geo},\cite{TAR},\cite{TarGEO},\cite{Tikho}, \cite{Auger},\cite{ZZ}. On the other hand, if $X_t\in\R$ describes the wealth process of an investor in the context of portfolio optimization, \cite{AN},\cite{Bielecki},\cite{BJ},\cite{Biswas},\cite{Kuper},\cite{PitStet}, then we are  rather interested in how fast the process $X_t$ grows. Let $W_t$, $t\in\N$,
denote the wealth process and let $\gamma<0$ denote the risk-averse parameter. Section \ref{SECA} shows that under general assumptions (any ergodicity-type assumptions on the log-wealth process are not involved) the following limit of the long-run risk sensitive criterion
$$C=\liminf\limits_{t\to\infty}\frac{1}{t}\frac{1}{\gamma}\log(E[(W_t)^\gamma])$$
determines the log-growth rate of the trajectories of  $W_t$ according to
$$\liminf\limits_{t\to\infty}\frac{1}{t}\log W_t\geq C\ a.s.$$
and that, if $C>0$, the barrier hitting times $T_b=\inf\{t\in\N\colon W_t>b\},\ b\in\R^+,$ are integrable and satisfy
$$\limsup\limits_{b\to+\infty}\frac{E[T_b]}{\log(b)}\leq\frac{1}{C}.$$

 \section{ Exponential convergence of random variables}\label{SEC}

In the whole paper we assume that $(\Omega,\Sigma,\PP)$ is a probability space and that $X_t\colon\Omega\to\R^+$, $t\in\mathbb{N}$, is a sequence of random variables, where $\R^+=[0,+\infty)$. We start with the following definition.
\begin{definition}
Given the sequence $x_t\in\R^+$, $t\in\N$,  define: 
$$\overline{R}(x_t)=\limsup\limits_{t\to\infty}\sqrt[\leftroot{3} t]{x_t}.$$
\end{definition}

We are mostly interested in the  convergence to $0$ and thus we will usually refer to  $\overline{R}(x_t)$  as convergence rate.  If  $x_t\to0$ then $\overline{R}(x_t)\in[0,1]$.  Condition $\overline{R}(x_t)\in (0,1)$ determines the exponential convergence of $x_t$ to zero. We will sometimes slightly abuse the nomenclature and we will refer to $\overline{R}(x_t)$ as convergence rate regardless of whether the $x_t$ converges or not. Condition $\overline{R}(x_t)=1$ excludes the exponential convergence and does not determine the convergence of $x_t$. Condition $\overline{R}(x_t)>1$ implies that some subsequence of $x_t$ diverges exponentially fast. The following  observation presents general characterisation of exponential convergence and is rather easy to prove.

\begin{observation}\label{O1}
 Given $C\in\R^+$ and a sequence $x_t\in \R^+$, the following conditions are equivalent:
\begin{enumerate}
\item $C\geq \limsup\limits_{t\to\infty} \sqrt[\leftroot{3} t]{x_t}$,
\item  $\log C\geq \limsup\limits_{t\to\infty} \frac{1}{t}\log(x_t)$, where $\log(\cdot)=\exp^{-1}(\cdot)$ and $\log(0):=-\infty$,
\item for any $R>C$,  $ \sup\limits_{t\in\N}\frac{x_t}{R^t}<+\infty$.
\item for any $R>C$,  $ \lim\limits_{t\to\infty}\frac{x_t}{R^t}=0$,
\item for any $R>C$, $\sum\limits_{t\in\N}\frac{x_t}{R^t}<\infty$,
\item for any $R>C$, $\sum\limits_{t\in\N}\sum\limits_{i=t}^{\infty}\frac{x_i}{R^i}<\infty$.
\end{enumerate}
If additionally $C\in (0,1)$ and  $\tau_{\eps}:=\inf\{t\in\mathbb{N}\colon x_t<\eps\},\ \eps>0,$ then any of the above conditions implies that

\begin{equation}
\limsup\limits_{\eps\to 0^+} \frac{\tau_{\eps}}{|\log(\eps)|}\leq\frac{1}{|\log(C)|}.
\end{equation}

\end{observation}

Observation \ref{O1} implies that the value of the limit $\overline{R}(x_t)=\limsup\limits_{t\to\infty} \sqrt[\leftroot{3} t]{x_t}$ is the infimum of constants $C\geq0$ such that for any $R>C$:
\begin{equation}\label{equ}
x_t\leq M_R\cdot R^t, \ t\in\N, \mbox{ where }  M_R:=\sup\limits_{t\in\N}\frac{x_t}{R^t}<\infty.
\end{equation}
 Equation \eqref{equ} is satisfied for any $R> \overline{R}(x_t)$.

Given any measurable sequence $X_t\geq0$, $t\in\N$, we may consider the random variable  $C\colon\Omega\to[0,+\infty]$ given by
\begin{equation}\label{CC}
C:=\overline{R}(X_t)= \limsup\limits_{t\to\infty} \sqrt[\leftroot{3} t]{X_t}.
\end{equation}
 Below, if $E[X_t]=\infty$  then we put $\sqrt[\leftroot{3} t]{\infty}:=\infty$.  Theorem \ref{T1} does not assume that $X_t\to0$ and may describe the divergence of $X_t$ to infinity.
 
\begin{theorem}\label{T1}
Let $X_t\geq0$. The convergence rate of the trajectories is bounded from above by the convergence rate of the expectations in sense 
$$\PP[ \limsup\limits_{t\to\infty} \sqrt[\leftroot{3} t]{X_t}\leq \limsup\limits_{t\to\infty} \sqrt[\leftroot{3} t]{E[X_t]}]=1.$$
\end{theorem}
\begin{proof}
For the proof it is enough to assume that $A:=\limsup\limits_{t\to\infty} \sqrt[\leftroot{3} t]{E[X_t]}<\infty$ and it is enough to show that for any $\eps>0$ 
$$\PP[\limsup\limits_{t\to\infty} \sqrt[\leftroot{3} t]{X_t}>A+\eps]=0.$$
Fix $\eps>0$ and note that $\limsup\limits_{t\to\infty} \sqrt[\leftroot{3} t]{X_t}=\lim\limits_{t\to\infty}\sup\limits_{i\geq t}\sqrt[\leftroot{3} i]{X_i}$ . We thus have
$$\PP[\limsup\limits_{t\to\infty} \sqrt[\leftroot{3} t]{X_t} >A+\eps]=\PP[\lim\limits_{t\to\infty}\sup\limits_{i\geq t}\sqrt[\leftroot{3} i]{X_i}>A+\eps].$$
By the monotonicity $\sup\limits_{i\geq t+1}\sqrt[\leftroot{3} i]{X_i}\leq\sup\limits_{i\geq t}\sqrt[\leftroot{3} i]{X_i}$, we have
$$ \{\lim\limits_{t\to\infty}\sup\limits_{i\geq t}\sqrt[\leftroot{3} i]{X_i}>A+\eps \}\subset\bigcap\limits_{t\in\mathbb{N}}\{\sup\limits_{i\geq t}\sqrt[\leftroot{3} i]{X_i}>A+\eps\}$$
and 

$$\PP[\lim\limits_{t\to\infty}\sup\limits_{i\geq t}\sqrt[\leftroot{3} i]{X_i}>A+\eps]\leq\lim\limits_{t\to\infty}\PP[\sup\limits_{i\geq t}\sqrt[\leftroot{3} i]{X_i}>A+\eps]=$$
$$=\lim\limits_{t\to\infty}\PP[\bigcup\limits_{i=t}^{+\infty}\{X_i>(A+\eps)^i\}]\leq\lim\limits_{t\to\infty}\sum\limits_{i=t}^{\infty}\PP[X_i>(A+\eps)^i].$$
By Chebyshev inequality 
$$ \lim\limits_{t\to\infty}\sum\limits_{i=t}^{\infty}\PP[X_i>(A+\eps)^i] \leq \lim\limits_{t\to\infty}\sum\limits_{i=t}^{\infty}\frac{E[X_i]}{(A+\eps)^i},$$
and by point 5. of Observation \ref{O1} (with $x_t:=E[X_t]$ and $R=A+\eps$),
$$\lim\limits_{t\to\infty}\sum\limits_{i=t}^{\infty}\frac{E[X_i]}{(A+\eps)^i}=0.$$
\end{proof}
\begin{conclusion}\label{C1}
 The convergence rate of the expectations $E[X_t]$ is bounded from below by the essential suppremum of $\overline{R}(X_t)$, i.e.
$$ \limsup\limits_{t\to\infty} \sqrt[\leftroot{3} t]{E[X_t]}\geq \inf\{\hat{C}\in\R\colon \PP[\overline{R}(X_t)\leq\hat{C}]=1\}.$$
\end{conclusion}
%

Now we will focus on the convergence of the expectations of the following hitting times 
$$\tau_\epsilon=\inf\{t\in\N\colon X_t<\eps\},\ \eps>0.$$
Assume from now on that $X_t\to0$ a.s. so $\overline{R}(X_t)=\limsup\limits_{t\to\infty} \sqrt[\leftroot{3} t]{X_t}\leq 1$  a.s. Let 
$$C:=ess\sup \overline{R}(X_t)$$ denote the essential supremum of $\overline{R}(X_t)$. By point 3. of Observation \ref{O1}  for any $R> C$ there is some random variable $M_R\colon\Omega\to\R^{+}$  with
$$X_{t}\leq M_R\cdot R^t\ a.s. $$
 If $C\in (0,1)$  then again by Observation \ref{O1},
$$\limsup\limits_{\eps\to0^+}\frac{\tau_\eps}{|\log(\eps)|}\leq \frac{-1}{\log(C)}\ a.s.$$
At the same time the expectation of $\tau_{\epsilon}$ may be infinite, see a simple example below.
\begin{example}
Let $\Omega=\{0,1,2,\dots\}$, $\Sigma=\mcl{P}(\Omega)$ be the family of all subsets of $\Omega$ and let $p_n:=\PP[n], n\in\Omega$. Assume that $X_t$ is the characteristic function of the complement of the set $\N_t:=\{0,1,\dots,t\}$, i.e.
$$X_{t}(n)=1_{\N\setminus\N_t}(n),\ n\in\N.$$
Clearly we have $X_t(n)=0\Leftrightarrow n\leq t$. Hence, $\limsup\limits_{t\to\infty} \sqrt[\leftroot{3} t]{X_t}=0$. At the same time, for $\eps\in (0,1)$ we have 
$$\tau_\eps(n)={n},\ n\in\Omega,$$
and
$$E[\tau_\eps]=\sum\limits_{t\in\N}\PP[\tau_\eps> t]=\sum\limits_{t\in\N}\PP[X_t=1]=\sum\limits_{t=0}^{\infty}\PP[\N\setminus\N_t]=\sum\limits_{t=0}^{\infty}\sum\limits_{n> t}p_n.$$
 The above sum may be infinite depending on how the probabilities $p_n$ are chosen. At the same time $\overline{R}(X_t)=0$ and  $\lim\limits_{\eps\to0^+}\frac{\tau_\eps}{|\log(\eps)|}=0$ a.s.
\end{example}

If the expectations $E[X_t]$ decrease exponentially fast then $E[\tau_\eps]$ is finite and Theorem \ref{T2} shows that the convergence rate of the expectations $\overline{R}(E[X_t])$ controls the convergence of $E[\tau_\eps]$ with $\eps\to0^+$.  

\begin{theorem}\label{T2}
If $A:=\overline{R}(E[X_t])=\limsup\limits_{t\to\infty} \sqrt[\leftroot{3} t]{E[X_t]}<1$ then 
$$\limsup\limits_{\eps\to0^+} \frac{E[\tau_\eps]}{|\log(\eps)|} \leq \frac{-1}{\log(A)}.$$
\end{theorem}
Above, if $A=0$ then $\frac{-1}{\log(A)}:=0$. The  proof of the above theorem is based on the following lemma which follows from Theorem \ref{T1} and Observation \ref{O1}.

\begin{lemma}\label{HC}
Assume that $A=\limsup\limits_{t\to\infty} \sqrt[\leftroot{3} t]{E[X_t]}<1$.  For any $R>A$
the following random variable
$$H_R:=\inf\{t\in\N\colon \sup\limits_{i\geq t} \frac{X_i}{R^i}\ \leq \ 1\}$$
has finite expectation.
\end{lemma}
\begin{proof}

Fix $R>A$. By Theorem \ref{T1} and by point 4. of Observation \ref{O1} we have $\PP[H_R<\infty]=1$. We need to show $E[H_R]<\infty$. We have:
$$E[H_R]=\sum\limits_{t\in\N}\PP[H_R>t].$$

 By definition of $H_R$ and by Chebyszev inequality,
$$\PP[H_R>t]=\PP[\bigcup\limits_{i=t}^{\infty}\{X_i>R^i\}]\leq\sum\limits_{i=t}^{\infty}\PP[X_i>R^i]\leq \sum\limits_{i=t}^\infty\frac{E[X_i]}{R^i}.$$
We have thus shown that 
$$E[H_R]\leq \sum\limits_{t\in\N}\sum\limits_{i=t}^\infty\frac{E[X_i]}{R^i}$$
and the above series is finite by point 6. of Observarion \ref{O1} applied to the sequence $x_t=E[X_t]$.
\end{proof}

\emph{Proof of Theorem \ref{T2}}. Fix $C\in (A,1)$. Under notation of Lemma \ref{HC}, for any $t\in\N$  on set $\{t\geq H_C\}$ we have $X_{t+i}\leq C^{t+i}$ for all $i\in\N$. In other words, for any $t\geq H_C$ we have $X_t\leq C^t$. At the same time, as $C<1$, note that for 
$$\hat{C}:= \frac{-1}{log(C)}$$ the ceiling $t=\lceil \hat{C}\cdot |
\log(\eps)| \rceil$ is the smallest natural number with $C^t\leq\eps$. This implies that if $ t\geq\max\{H_C,\lceil \hat{C}\cdot |\log(\eps)| \rceil \}$ 
then $X_t\leq\eps$ and $X_{t+1}<\eps$. Hence, 
$$\tau_\eps\leq \max\{H_C, \lceil \hat{C}\cdot |\log(\eps)| \rceil  \}+1$$ and

\begin{equation}\label{tauH}
\frac{\tau_\eps}{|\log(\eps)|}\leq\frac{1}{|\log(\eps)|}(\max\{H_C, \lceil \hat{C}\cdot |\log(\eps)| \rceil  \}+1).
\end{equation}
 Thus
\begin{equation}\label{tauHH}
E[\frac{\tau_\eps}{|\log(\eps)|}]\leq\frac{1}{|\log(\eps)|}E[\max\{H_C, \lceil \hat{C}\cdot |\log(\eps)| \rceil  \}]+\frac{1}{|\log(\eps)|}.
\end{equation}
 Lemma \ref{HC} allows us to use the  Lebesgue's dominated convergence theorem:
$$\frac{1}{|\log(\eps)|}E[\max\{H_C, \lceil \hat{C}\cdot |\log(\eps)| \rceil  \}]=E[\max\{\frac{H_C}{|\log(\eps)|},\hat{C}\}]\to\hat{C} \mbox{ with } \eps\to0^+.$$
Equation \eqref{tauHH} and the above imply that $\limsup\limits_{\eps\to0^+}\frac{E[\tau]}{|\log(\eps)|}\leq\hat{C}\mbox{ for any } C\in (A,1)$ and thus
 $$\limsup\limits_{\eps\to0^+}\frac{E[\tau_\eps]}{|\log(\eps)|}\leq\hat{A}=\frac{-1}{\log(A)}.$$ \hfill{$\Box$}

\section{Applications.}\label{SECA}
The results of the previous sections will be applied to some problems of optimization, control and estimation.
\subsection{Optimization}
Let $(K,d)$ be some metric space and $f\colon K\to\R^+$ be a Borel-measurable function which attains its global minimum $f_{\min}$. Assume that $X_t\in K$, $t\in\N$, represents an optimization process. The hitting time of the $\eps$-optimal sublevel set $\{x\in A\colon f(x)<f_{\min}+\eps\}$ is defined by
$$\tau_{\eps}=\inf\{t\in\N\colon f(X_t)<f_{\min}+\epsilon\}.$$
Theorems \ref{T1} and \ref{T2} give us the relations between convergence rate of the trajectories $f(X_t)$, the convergence rate of the expectations $E[|f(X_t)-f_{\min}|]$ and the convergence behaviour of $E[\tau_\eps]$ under general assumptions on the sequence $f(X_t)$. In particular, if the following constant 
$$A=\overline{R}(E[f(X_t)-f_{\min}],t\in\N)=\limsup\limits_{t\to\infty}\sqrt[\leftroot{3} t]{E[f(X_t)-f_{\min}]}$$ 
satisfies  $A<1$ then by Theorem \ref{T2} we have the following upper bound $\limsup\limits_{t\to\infty}\frac{E[\tau_\eps]}{|\log(\eps)|}\leq\frac{-1}{\log(A)}$ and the following control of the asymptotic behaviour of $E[\tau_\eps]$ with $\eps\to0^+$ :
$$E[\tau_\eps]\leq C(\eps)\cdot |\log(\eps)|,$$ 
where  $C\colon (0,1)\to \R^+$ satisfies $\limsup\limits_{\eps\to0^+}C(\eps)\leq\hat{A}=\frac{-1}{\log(A)}$. This  strengthens Theorem 7 from \cite{TAR}. We also have, by \eqref{E1} and \eqref{E3},
$$\PP[\limsup\limits_{t\to\infty}\sqrt[\leftroot{3} t]{f(X_t)-f_{min}}\leq A]=1 \mbox { and } 
\lim\limits_{\eps\to0^+}\limsup\limits_{t\to\infty} \sqrt[\leftroot{3} t]{\PP[f(X_t)\geq f_{\min}+\eps]}\leq 
A. $$

\subsection{Risk sensitive control} Let $W_t\colon\Omega\to\R^+$, $t\in\N$, represent the wealth process of an investor, see \cite{AN},\cite{Bielecki},\cite{BJ},\cite{Biswas},\cite{PitStet} for details. Let $\gamma\neq0$ represent the risk-averse parameter of an investor so we are interested in the limit of the long-run risk-sensitive criterion of the following log wealth growth:

\begin{equation}\label{CCCC}
C=\liminf\limits_{t\to\infty}\frac{1}{t}\frac{1}{\gamma}\log(E[(W_t)^\gamma]).
\end{equation}

We will assume that $\gamma<0$ which is a common investment criterion in the context of long run optimization. This section shows how the value of the risk-sensitive criterion determines the exponential growth of the trajectories of the wealth process $W_t$ and the logarithmic growth of the expected barrier-hitting time $T_b$ with $b\nearrow\infty$. This type of convergence behaviour is natural under suitable ergodicity-type assumptions on the log-wealth process, see for instance the asymptotic optimality principle, \cite{AN}. This section brings to attention that the log  growth rate of the trajectories of $W_t$ and the logarithmic growth of the corresponding expected hitting times are determined by the limit $\eqref{CCCC}$ under general assumptions on $W_t$ (no ergodicity  involved) which follows from the results of Section \ref{SEC}.

\begin{theorem}\label{BAR}
 Let $\gamma<0$ and let $W_t\colon\Omega\to(0,\infty)$ satisfy $E[(W_t)^\gamma]<\infty$, $t\in\N$. The constant
$$C=\liminf\limits_{t\to\infty}\frac{1}{t}\frac{1}{\gamma}\log(E[(W_t)^\gamma])$$
determines the log-growth of $W_t$ according to
$$\liminf\limits_{t\to\infty}\frac{1}{t}\log W_t\geq C\ a.s.$$
Additionally, if $C>0$ then the following hitting times $T_b=\inf\{t\in\N\colon W_t>b\},\ b\in\R^+,$ are integrable and satisfy
$$\limsup\limits_{b\to+\infty}\frac{E[T_b]}{\log(b)}\leq\frac{1}{C}.$$
\end{theorem}

Let $X_t:=\frac{1}{W_t}$. Recall that 
\begin{equation}\label{XW}E[(X_t)^{p_1}]^{\frac{1}{p_1}}\leq E[(X_t)^{p_2}]^{\frac{1}{p_2}}\mbox { for }0<p_1<p_2<\infty,\end{equation}
assuming that the above expectations exist.   For $\gamma<0$ let
$$C_\gamma=\liminf\limits_{t\to\infty}\frac{1}{t}\frac{1}{\gamma}\log(E[(W_t)^\gamma]).$$
By  elementary calculations and by \eqref{XW} one may show that  
\begin{equation}\label{Cgamma}
C_{\gamma_1}\leq C_{\gamma_2}\mbox{ for }\gamma_1 <\gamma_2<0.
\end{equation}
See also Lemma 2.1 in \cite{Kuper} for the above. Theorem \ref{BAR} and Equation \eqref{Cgamma} immediately lead to the following.

\begin{conclusion}
 Assume that $W_t\colon\Omega\to(0,\infty)$ satisfy $E[(W_t)^\gamma]<\infty$, $t\in\N$, for any $\gamma<0$ close to zero. The following limit
$$C=\lim\limits_{\gamma\to0^-}\liminf\limits_{t\to\infty}\frac{1}{t}\frac{1}{\gamma}\log(E[(W_t)^\gamma])$$
bounds the log-growth of $W_t$ according to $\liminf\limits_{t\to\infty}\frac{1}{t}\log W_t\geq C\ a.s.$
Additionally, if $C>0$ then the  hitting times $T_b=\inf\{t\in\N\colon W_t>b\}$ are integrable and satisfy
$$\limsup\limits_{b\to+\infty}\frac{E[T_b]}{\log(b)}\leq\frac{1}{C}.$$
\end{conclusion}\textbf{Proof of Theorem \ref{BAR}}. As $\gamma<0$, equation \eqref{CCCC} implies that
 $$C\cdot\gamma=\limsup\limits_{t\to\infty}\frac{1}{t}\log(E[(W_t)^\gamma]).$$
Hence, \begin{equation}\label{Aqua}\exp(C\gamma)=\limsup\limits\limits_{t\to\infty} (E[(W_t)^\gamma])^\frac{1}{t}.\end{equation} By Theorem \ref{T1},
$$\limsup\limits_{t\to\infty}((W_t)^\gamma)^\frac{1}{t}\leq\exp(C\cdot\gamma)\ a. s.$$
As $\gamma<0$, the above determines the growth of $W_t$ according to 
$$\liminf\limits_{t\to\infty}((W_t)^{|\gamma|})^\frac{1}{t}\geq\exp(C\cdot|\gamma|) \ a. s.$$
 Hence,  we have
 \begin{equation}\label{WW}
\liminf\limits_{t\to\infty}(W_t)^\frac{1}{t}\geq\exp(C)\ a. s.
\end{equation}
and $$\liminf\limits_{t\to\infty}\frac{1}{t}\log W_t\geq C\ a.s.$$
which proves the first part of the theorem.

Now we assume that $C>0$ and we will discuss the hitting times $T_b=\inf\{t\in\N\colon W_t>b\}.$
 Let
$$\tau_{\eps}=\inf\{t\in\N\colon (W_t)^\gamma<\eps\}=\inf\{t\in\N\colon W_t>(\tfrac1\eps)^{|\gamma|^{-1}}\}.$$
Condition $C>0$  implies $\exp(\gamma\cdot C)<1$. By \eqref{Aqua} and Theorem \ref{T2}, we have
\begin{equation}\label{Tosia}\limsup\limits_{\eps\to0^+}\frac{E[\tau_\eps]}{|\log(\eps)|}\leq\frac{1}{C\cdot|\gamma|}.\end{equation}
Let 
$$T(\eps)=\inf\{t\in\N\colon W_t>\frac{1}{\eps}\}$$
so we have $\tau_\eps=\inf\{t\in\N\colon W_t>\frac{1}{\eps^{|\gamma|^{-1}}}\}=  T(\eps^{|\gamma|^{-1}})$ and 
$$\limsup\limits_{\eps\to0^+}\frac{E[\tau_\eps]}{|\log(\eps)|}=\limsup\limits_{\eps\to0^+}\frac{E[T(\eps^{\frac{1}{|\gamma|}})]}{|\log(\eps)|}=\limsup\limits_{\eps\to0^+}\frac{E[T(\eps)]}{|\log(\eps^{|\gamma|})|}=\limsup\limits_{\eps\to0^+}\frac{E[T(\eps)]}{|\gamma|\cdot|\log(\eps)|}.$$
Hence, by \eqref{Tosia}, $\limsup\limits_{\eps\to0^+}\frac{E[T(\eps)]}{|\gamma|\cdot|\log(\eps)|}\leq\frac{1}{|\gamma|\cdot C}$, and thus
$$ \limsup\limits_{\eps\to0^+}\frac{E[T(\eps)]}{|\log(\eps)|}\leq\frac{1}{C}.$$
For $b=\frac{1}{\eps}$ we have $T_b=T(\eps)$ and $|\log(\eps)|=\log(\frac1\eps)=\log(b)$. We thus have
$$\limsup\limits_{b\to+\infty}\frac{E[T_b]}{\log(b)}\leq\frac{1}{C}.$$ \hfill{$\Box$}

\subsection{Estimation} Assume that a sequence $Z_t\in\R$ approaches the unknown parameter $\theta\in\Theta\subset\R$ and that the mean squared error $E|Z_t-\theta|^2$ satisfies 
$$C=\limsup\limits_{t\to\infty}(E|Z_t-\theta|^2)^\frac{1}{t}.$$
The results of Section \ref{SEC} can be applied to the sequence $X_t=|Z_t-\theta|^2$. In particular, by Theorem \ref{T1}, we bound the convergence rate of the trajectories
$$\PP[\limsup\limits_{t\to\infty} (|Z_t-\theta|^2)^\frac{1}{t}\leq C]=1\mbox { and hence } \PP[\limsup\limits_{t\to\infty} (|Z_t-\theta|)^\frac{1}{t}\leq \sqrt{C}]=1.$$ 
More generally, if $(K,d)$ is a metric space, $\theta\in K$ and $Z_t\in K$ converges in mean to $\theta$, i.e. $E[d(Z_t,\theta)]\to0$, then   the results of Section \ref{SEC} may be applied to the sequence $X_t:=d(Z_t,\theta)$.

%

{}

\begin{thebibliography}{99} 
\bibitem{Agapie} Agapie, A. (2022). Evolution strategies under the 1/5 success rule. Mathematics, 11(1), 201.
\bibitem{AN} Algoet, P. H., and Cover, T. M. (1988). Asymptotic optimality and asymptotic equipartition properties of log-optimum investment. The Annals of Probability, 876-898.
\bibitem{Bielecki} Bielecki, T. R., and Pliska, S. R. (2003). Economic properties of the risk sensitive criterion for portfolio management. Review of Accounting and Finance, 2(2), 3-17.
\bibitem{BJ} Bäuerle, N., and Jaśkiewicz, A. (2024). Markov decision processes with risk-sensitive criteria: an overview. Mathematical Methods of Operations Research, 99(1), 141-178.

\bibitem{Biswas} Biswas, A., and Borkar, V. S. (2023). Ergodic risk-sensitive control—a survey. Annual Reviews in Control, 55, 118-141.
\bibitem{Markov} Douc, R., Moulines, E., Priouret, P., and Soulier, P. (2018). Markov chains, Cham, Switzerland: Springer International Publishing.
\bibitem{Dud} Dudley, R. M. (2018). Real analysis and probability. CRC Press.
\bibitem{Jager} J{\"{o}}gersk{\"{u}}pper, J., (2007). Algorithmic analysis of a basic evolutionary algorithm for continuous optimization. Theoretical Computer Science, 379(3), 329-347.
\bibitem{Jager2} J{\"{o}}gersk{\"{u}}pper, J., (2006). "How the (1+ 1) ES using isotropic mutations minimizes positive definite quadratic forms." Theoretical Computer Science 361.1 , 38-56.
\bibitem{Kuper} Kupper, M., and Schachermayer, W. (2009). Representation results for law invariant time consistent functions. Mathematics and Financial Economics, 2(3), 189-210.
\bibitem{Liu} Liu, Q., and Watbled, F. (2009). Exponential inequalities for martingales and asymptotic properties of the free energy of directed polymers in a random environment. Stochastic processes and their applications, 119(10), 3101-3132.

\bibitem{Mori} Morinaga, D., Fukuchi, K., Sakuma, J., and Akimoto, Y. (2023). Convergence Rate of the (1+ 1)-ES on Locally Strongly Convex and Lipschitz Smooth Functions. IEEE Transactions on Evolutionary Computation.
\bibitem{PitStet} Pitera, M., and Stettner, Ł. (2023). Discrete‐time risk sensitive portfolio optimization with proportional transaction costs. Mathematical Finance, 33(4), 1287-1313.

\bibitem{R1} Rudolph, G. (1994, June). Convergence of non-elitist strategies. In Proceedings of the First IEEE Conference on Evolutionary Computation. IEEE World Congress on Computational Intelligence (pp. 63-66). IEEE.
\bibitem{geo}G. Rudolph, (1997). Convergence rates of evolutionary algorithms for a class of convex objective functions. Control and Cybernetics Volume 26 Issue 3 , pp. 375-390.
\bibitem{EC} Sun, J. (1997). Exponential convergence for sequences of random variables. Statistics and probability letters, 34(2), 159-164.
\bibitem{TAR} Tar{\l}owski, D. (2024). On asymptotic convergence rate of random search. Journal of Global Optimization, 89(1), 1-31.
\bibitem{TarGEO} Tar{\l}owski, D. (2018). On geometric convergence rate of Markov search towards the fat target. Operations Research Letters, 46(1), 33-36.
\bibitem{Tikho} Tikhomirov, A. S. (2019, October). On the convergence rate of the quasi–Monte Carlo method of search for extremum. In Journal of Physics: Conference Series (Vol. 1352, No. 1, p. 012051). IOP Publishing.
\bibitem{Auger} Toure, C., Auger, A.  Hansen, N. (2023). Global linear convergence of evolution strategies with recombination on scaling-invariant functions. Journal of Global Optimization, 86(1), 163-203.

\bibitem{ZZ} Zhigljavsky, A., and Zilinskas, A. (2021). Bayesian and high-dimensional global optimization. Cham: Springer International Publishing.
\end{thebibliography}
\end{document}